\documentclass[11pt, a4paper, twoside]{article}

\usepackage[
  top=1.25in, bottom=1.25in,
  left=1.25in, right=1.25in
]{geometry}

\usepackage[T1]{fontenc}
\usepackage[sc,osf]{mathpazo}
\usepackage{eulervm}
\usepackage[scaled=0.85]{inconsolata}

\usepackage[final]{microtype}

\usepackage{amsmath, amssymb, amsthm}
\usepackage{mathtools}
\usepackage{bm}

\usepackage[dvipsnames]{xcolor}
\usepackage[
  colorlinks = true,
  linkcolor  = RoyalBlue,
  citecolor  = ForestGreen,
  urlcolor   = RoyalBlue
]{hyperref}

\usepackage{aliascnt}

\newtheorem{theorem}{Theorem}[section]

\newcommand{\makealiasenv}[2]{%
  \newaliascnt{#1}{theorem}%
  \newtheorem{#1}[#1]{#2}%
  \aliascntresetthe{#1}%
}

\makealiasenv{lemma}{Lemma}
\makealiasenv{proposition}{Proposition}
\makealiasenv{corollary}{Corollary}
\makealiasenv{conjecture}{Conjecture}
\makealiasenv{claim}{Claim}

\theoremstyle{definition}
\makealiasenv{definition}{Definition}
\makealiasenv{example}{Example}
\makealiasenv{exercise}{Exercise}
\makealiasenv{problem}{Problem}
\makealiasenv{construction}{Construction}

\theoremstyle{remark}
\makealiasenv{remark}{Remark}
\makealiasenv{notation}{Notation}
\makealiasenv{observation}{Observation}

\newtheorem*{theorem*}{Theorem}
\newtheorem*{lemma*}{Lemma}
\newtheorem*{remark*}{Remark}

\usepackage[capitalise, noabbrev]{cleveref}

\crefname{theorem}{Theorem}{Theorems}
\crefname{lemma}{Lemma}{Lemmas}
\crefname{proposition}{Proposition}{Propositions}
\crefname{corollary}{Corollary}{Corollaries}
\crefname{conjecture}{Conjecture}{Conjectures}
\crefname{claim}{Claim}{Claims}
\crefname{definition}{Definition}{Definitions}
\crefname{example}{Example}{Examples}
\crefname{exercise}{Exercise}{Exercises}
\crefname{remark}{Remark}{Remarks}
\crefname{equation}{Equation}{Equations}
\crefname{section}{Section}{Sections}
\crefname{figure}{Figure}{Figures}
\crefname{table}{Table}{Tables}

\usepackage[square, numbers, sort&compress]{natbib}
\usepackage{graphicx}
\usepackage{booktabs}
\usepackage{enumitem}
\usepackage{csquotes}
\usepackage{stmaryrd}
\usepackage{authblk}
\usepackage{tikz}
\usepackage{fancyhdr}
\usepackage{titlesec}
\titleformat{\section}{\large\bfseries\scshape}{\thesection.}{0.6em}{}
\titleformat{\subsection}{\normalsize\bfseries}{\thesubsection.}{0.5em}{}

\usepackage{abstract}

\newcommand{\R}{\mathbb{R}}

\newcommand{\esssup}{\operatorname*{ess\,sup}}

\begin{document}

\title{\textbf{\Large GLOBAL UCP FOR PARABOLIC FRACTIONAL $p$-LAPLACE EQUATION WITH VERY ROUGH POTENTIALS}}

\author[]{HARSH PRASAD}
\affil[]{Fakult\"{a}t f\"{u}r Mathematik, Universit\"{a}t Bielefeld\\
  \texttt{hprasad@math.uni-bielefeld.de}}

\date{\today}
\maketitle
\pagestyle{fancy}
\fancyhf{}
\fancyhead[C]{%
  \ifodd\value{page}%
    GLOBAL UCP %
  \else%
    PRASAD%
  \fi%
}
\fancyfoot[C]{\thepage}
\renewcommand{\headrulewidth}{0pt}
\begin{abstract}
We show that the global unique continuation principle holds for the parabolic fractional $p-$Laplace equation with very rough potentials $V(x,t) \in L^{p'}_tW^{-s,p'}_x$. Whereas the result is new even for the fractional $p-$Laplace operator, the corresponding local problem remains open even with zero potential. The short proof eschews extension techniques and Carleman estimates. 
\end{abstract}



\tableofcontents

\section{Introduction}
\label{sec:intro}
We are interested in studying unique continuation for weak solutions to the following non-linear, nonlocal parabolic equation:
\begin{equation}\label{eq:main-eq}
    \partial_t u + (-\Delta_p)^s u + V(x,t)u(x,t) = 0, \quad \text{in } \Omega \times I,
\end{equation}
where $\Omega \subset \R^N$ is open, $I \subset \R$ is an open
interval. 
\subsection{Background and Novelty}
Unique continuation property (ucp) for the fractional Laplace was first studied by M. Riesz \cite{riesz1937integrales} using the Kelvin transform. In recent years, there has been a surge of interest in studying unique continuation properties of nonlocal operators \cite{fall2014unique,ruland2015unique, fall2015unique, ruland2017quantitative, seo2015unique, yu2017unique, ruland_unique_2019}. Besides their intrinsic interest they are an important tool in studying the stability of the Cauchy problem \cite{tataru1996carleman,tataru2004unique}, in studying nodal sets \cite{koch2016variable, soave2018nodal, soave2018nodal_sublinear}, in studying inverse problems \cite{ghosh2020calderon, ghosh2020uniqueness, bhattacharyya2018inverse, ghosh2017calderon, ghosh2018uniqueness}. 

The main technique for establishing these properties have been via studying Carleman estimates for certain degenerate elliptic equations which arise as the Caffarelli-Silvestre \cite{caffarelli2007extension} extensions of the nonlocal operators under consideration. Indeed, in the case of the fractional heat equation, to the best of our knowledge, there is no ucp result and the only available results are for fractional powers of the heat operator which again proceed via delicate Carleman estimates for an appropriate extension problem \cite{banerjee_calderon_2024, banerjee_decay_2025, banerjee_extension_2024, banerjee_unique_2024, arya_quantitative_2023, arya_space-like_2023, arya_space-like_2023-1}. 

For the $p-$Laplace equation
\[
-\text{div}(|\nabla u|^{p-2}u)=0
\]
unique continuation remains an open problem for all dimensions larger than $2$; in two dimension it was shown in \cite{Giovanni1987, manfredi1988, bojarski1987} and the case for rough potentials or any result for the parabolic $p-$Laplace remains out of reach. 

In this context, Berger and Schilling \cite{berger_unique_2026} recently established a characterisation for ucp for Levy operators which yielded an elementary proof for the global ucp for the fractional Laplace operator in a functional analytic framework. We directly generalize their idea to weak solutions of nonlinear, nonlocal parabolic equations with rough potentials (cf.~\cref{rem:gen}.)

\subsection{Main Theorem}
We state the main theorem below. We note that we get a global ucp in that the solution vanishes in all of $\R^N$; such a result is false in the local case and indeed fails even for the Laplace operator. In that sense, the result is an example of a purely nonlocal phenomenon.  
\begin{theorem}[Global Unique Continuation Principle]\label{thm:ucp}
Let $u$ be a weak solution to \eqref{eq:main-eq} in
$\Omega \times I$.  If there exist a non-empty open subset
$G \subset \Omega$ and a subinterval $J \subset I$ such that
$u(x,t) = 0$ for almost every $(x,t) \in G \times J$, then
$u(x,t) = 0$ for almost every $(x,t) \in \mathbb{R}^N \times J$.
\end{theorem}



\section{Preliminaries}
\label{sec:prelim}

\subsection*{Notation}

We write $B_R(x_0) := \{x \in \R^N : |x - x_0| < R\}$. 
\subsection*{Function spaces}

\paragraph{Fractional Sobolev spaces.}
Let $E \subset \R^N$ be an open set,  $1<p<\infty$ and $s \in (0,1)$. The fractional Sobolev space
$W^{s,p}(E)$ consists of all $u \in L^p(E)$ such that
\[
  [u]_{W^{s,p}(E)}^p
  := \int_E \int_E
  \frac{|u(x)-u(y)|^p}{|x-y|^{N+sp}}\, dx\, dy < \infty.
\]
It is a reflexive Banach space under the norm $\|u\|_{W^{s,p}(E)} :=
\|u\|_{L^p(E)} + [u]_{W^{s,p}(E)}$. The local
space $W^{s,p}_{\mathrm{loc}}(E)$ consists of all $u$
with $u \in W^{s,p}(U)$ for every open $U \Subset E$.

\paragraph{Tail space.}
The tail space $L^{p-1}_{sp}(\R^N)$ is defined by
\[
  L^{p-1}_{sp}(\R^N)
  := \Bigl\{ u \in L^{p-1}_{\mathrm{loc}}(\R^N) :
     \int_{\R^N}
     \frac{|u(x)|^{p-1}}{1+|x|^{N+sp}}\, dx < \infty
     \Bigr\}.
\]
For $u(\cdot,t) \in L^{p-1}_{sp}(\R^N)$, $x_0 \in \R^N$
and $r > 0$, we set
\[
  \mathrm{Tail}(u(\cdot,t);\, x_0, r)
  := \Bigl( r^{sp}
     \int_{\R^N \setminus B_r(x_0)}
     \frac{|u(x,t)|^{p-1}}{|x-x_0|^{N+sp}}\, dx
     \Bigr)^{\!\frac{1}{p-1}}.
\]
For a time interval $I = (t_1, t_2)$ we define the
parabolic tail
\[
  \mathrm{Tail}_\infty(u;\, x_0, r, I)
  := \esssup_{t \in I}\,
     \mathrm{Tail}(u(\cdot,t);\, x_0, r).
\]

\subsection*{Definition of weak solution}

\begin{definition}\label{def:weak-solution}
Let $\Omega \subset \mathbb{R}^N$ be open, $I = (t_1, t_2)$,
$1 < p < \infty$, and $s \in (0,1)$.  Let $V \in L^{p'}_{\mathrm{loc}}\!\bigl(I;\,W^{-s,p'}_{\mathrm{loc}}(\Omega)\bigr)$. A function
\[
  u \in L^p_{\mathrm{loc}}\!\bigl(I;\,W^{s,p}_{\mathrm{loc}}(\Omega)\bigr)
  \cap L^\infty_{\mathrm{loc}}\!\bigl(I;\,L^{p-1}_{sp}(\mathbb{R}^N)\bigr)
  \cap C_{\mathrm{loc}}\!\bigl(I;\,L^2_{\mathrm{loc}}(\Omega)\bigr)
\]
is a \emph{local weak solution} to \eqref{eq:main-eq} if, for
every $\varphi \in C^\infty_c(\Omega \times I)$,
\begin{equation}\label{eq:weak-form}
\begin{aligned}
  &-\iint_{\Omega \times I} u(x,t)\,\partial_t\varphi(x,t)\,dx\,dt \\
  &\quad
   + \int_I \iint_{\mathbb{R}^N\times\mathbb{R}^N}
     \frac{|u(x,t)-u(y,t)|^{p-2}(u(x,t)-u(y,t))
           (\varphi(x,t)-\varphi(y,t))}
          {|x-y|^{N+sp}}\,dx\,dy\,dt \\
  &\quad
   + \iint_{\Omega \times I} V(x,t)\,u(x,t)\,\varphi(x,t)\,dx\,dt
   = 0.
\end{aligned}
\end{equation}
\end{definition}


\section{Proof of Main Theorem}

\begin{proof}
\textbf{Step~1 (Reduction to a convolution condition):}
Let $\varphi \in C_c^\infty(G \times J)$ be an arbitrary test
function compactly supported in $G \times J$.
Because $u = 0$ a.e.\ on $G \times J$, both local terms in the
weak formulation \eqref{eq:weak-form} vanish:
\[
  \iint_{G \times J} u\,\partial_t\varphi\,dx\,dt = 0
  \qquad \text{and} \qquad
  \iint_{G \times J} V\,u\,\varphi\,dx\,dt = 0.
\]
We are left with the nonlocal double integral.  Set
$\gamma := N + sp$.  Decomposing
$\mathbb{R}^N \times \mathbb{R}^N$ into the four disjoint regions
$G \times G$, $G^c \times G^c$, $G \times G^c$, and
$G^c \times G$, and using $u = 0$ on $G$ together with the
support of $\varphi$ being contained in $G$, we obtain
\begin{equation}\label{eq:reduced_integral}
  -2\int_J \int_G \varphi(x,t)
    \left(\int_{\mathbb{R}^N \setminus G}
      \frac{|u(y,t)|^{p-2}u(y,t)}{|x-y|^{\gamma}}\,dy
    \right)dx\,dt = 0.
\end{equation}
Since $\varphi \in C_c^\infty(G \times J)$ was arbitrary, it
follows that for a.e.\ $(x,t) \in G \times J$,
\begin{equation}\label{eq:pointwise_ucp}
  \int_{\mathbb{R}^N \setminus G}
    \frac{|u(y,t)|^{p-2}u(y,t)}{|x-y|^{\gamma}}\,dy = 0.
\end{equation}

\medskip
\textbf{Step~2 (Zeroth moment vanishes):}
Define $f_t(y) := |u(y,t)|^{p-2}u(y,t)$.  By a Fubini--Tonelli
argument we may fix a \emph{common} time $t \in J$ such that
\eqref{eq:pointwise_ucp} holds for a.e.\ $x \in G$ and
$u(\cdot,t) = 0$ a.e.\ in $G$.

Since $G$ is open and non-empty, we may choose $x_0 \in G$ and
$\epsilon > 0$ with $B_{\epsilon}(x_0) \Subset G$.  By translating
the coordinate system so that $x_0 = 0$ (which leaves the kernel
$|x-y|^{-\gamma}$ and the Lebesgue measure invariant), there 
is no loss of generality in assuming that $0 \in G$ and
$B_\epsilon(0) \Subset G$.

We claim that
\begin{equation}\label{eq:convolution_zero}
  \Phi(x) := \int_{\mathbb{R}^N \setminus G}
    \frac{f_t(y)}{|x-y|^{\gamma}}\,dy = 0
  \quad \text{for every } x \in B_\epsilon(0).
\end{equation}
Indeed, for $x \in B_\epsilon(0)$ and
$y \in \mathbb{R}^N \setminus G$ we have $|x-y| \ge \epsilon$, so
the map $x \mapsto |x-y|^{-\gamma}$ is bounded and continuous in
$x$ uniformly in $y$.  Since $f_t \in L^1_{\mathrm{loc}}$ and the
tail condition $u \in L^\infty_{\mathrm{loc}}(I; L^{p-1}_{sp}(\mathbb{R}^N))$
gives $\int |f_t(y)|(1+|y|^{\gamma})^{-1}\,dy < \infty$, the map
$\Phi$ is continuous on $B_\epsilon(0)$ by dominated convergence.
The claim now follows from \eqref{eq:pointwise_ucp}.


\medskip
\textbf{Step~3 (All moments vanish):}
Because $|x-y| \ge \epsilon > 0$ for all $x \in B_\epsilon(0)$
and $y \in \mathbb{R}^N \setminus G$, the kernel
$|x-y|^{-\gamma}$ is smooth in $x$ on $B_\epsilon(0)$.
Moreover, for every multi-index $\beta \in \mathbb{N}_0^N$ there
exists a constant $C_\beta > 0$ such that
\[
  \bigl|\partial_x^\beta |x-y|^{-\gamma}\bigr|
  \le C_\beta \, |y|^{-\gamma - |\beta|}
  \quad \text{for all } x \in B_\epsilon(0),\;
  |y| \ge 2\epsilon,
\]
and the function $y \mapsto |f_t(y)| \cdot |y|^{-\gamma-|\beta|}$
is integrable over $\mathbb{R}^N \setminus G$ (again by the tail
condition on $u$.)  Dominated convergence therefore permits
differentiating under the integral sign to any order, giving
$\partial_x^\beta \Phi(x) = 0$ for all $x \in B_\epsilon(0)$.

We now compute $\partial_x^\beta |x-y|^{-\gamma}$ at $x = 0$. We proceed
via induction on $|\beta|$.  The base case $|\beta|=0$ yields
$|y|^{-\gamma}$.  For the inductive step, differentiating
$\partial_{x_i}(|x-y|^{-\gamma}) = \gamma(y_i - x_i)|x-y|^{-\gamma-2}$
and applying the product rule shows that every term produced by a
further differentiation is of the form
\[
  c_{\alpha,m}\; y^\alpha \,|y|^{-\gamma - 2m},
  \qquad
  |\alpha| \le |\beta|,\quad
  |\alpha| \equiv |\beta| \pmod{2},\quad
  m = \tfrac{|\beta|+|\alpha|}{2}.
\]
Hence, evaluating $\partial_x^\beta \Phi(0) = 0$ yields
\begin{equation}\label{eq:moment_condition}
  \int_{\mathbb{R}^N \setminus G}
    f_t(y)\, y^\alpha |y|^{-\gamma - 2m}\,dy = 0
\end{equation}
for every $\alpha \in \mathbb{N}_0^N$ and $m \ge 0$ with
$|\alpha| \equiv |\beta| \pmod{2}$ and
$m = \tfrac{|\beta|+|\alpha|}{2}$.

\medskip
\textbf{Step~4 (A density argument):}
Consider the algebra
\[
  \mathcal{B}
  := \operatorname{span}\Bigl\{
       y^\alpha |y|^{-2|\alpha|}
     \;\Big|\;
       \alpha \in \mathbb{N}_0^N,\; |\alpha| \ge 1
     \Bigr\}
  \subset C_0\!\bigl(\mathbb{R}^N \setminus B_\epsilon(0)\bigr).
\]
We note that each generator $y_i/|y|^2$ belongs to $C_0$ and since
$y^\alpha|y|^{-2|\alpha|}\cdot y^\beta|y|^{-2|\beta|}
= y^{\alpha+\beta}|y|^{-2(|\alpha|+|\beta|)}$,
$\mathcal{B}$ is also closed under products. Thus, $\mathcal{B}$ is a 
subalgebra of $C_0(\mathbb{R}^N \setminus B_\epsilon(0))$. Next, we have that $\mathcal{B}$, 
\begin{enumerate}
  \item \emph{separates points}: if $y \ne z$ in
        $\mathbb{R}^N \setminus B_\epsilon(0)$, then we can find an index $i$ such that
        $y_i/|y|^2 \ne z_i/|z|^2$ and 
  \item \emph{vanishes nowhere}: for each
        $y \in \mathbb{R}^N \setminus B_\epsilon(0)$ there exists
        some $y_i \ne 0$, so $y_i/|y|^2 \ne 0$ at $y$.
\end{enumerate}
Therefore $\mathcal{B}$ is dense in
$C_0(\mathbb{R}^N \setminus B_\epsilon(0))$ by the Stone–Weierstrass theorem. 

Now define the signed measure
$d\mu_t(y) := f_t(y)\,|y|^{-\gamma}\,dy$ on
$\mathbb{R}^N \setminus G \subset \mathbb{R}^N \setminus B_\epsilon(0)$.
The tail condition shows that
$\int |f_t(y)|\,|y|^{-\gamma}\,dy < \infty$, so $\mu_t$ is a
finite signed measure.  Factoring $|y|^{-\gamma}$ from
\eqref{eq:moment_condition} shows that $\mu_t$ annihilates every
element of $\mathcal{B}$, hence (by density) every element of
$C_0(\mathbb{R}^N \setminus B_\epsilon(0))$.  Thus, by the Riesz
representation theorem, $\mu_t = 0$, i.e.\
\[
  f_t(y)\,|y|^{-\gamma} = 0
  \quad \text{for a.e.\ } y \in \mathbb{R}^N \setminus G.
\]
It follows that $u(y,t) = 0$, for a.e.\ $y \in \mathbb{R}^N \setminus G$.
Recalling that $u(\cdot,t) = 0$ a.e.\ in $G$, we
obtain $u(\cdot,t) = 0$ a.e.\ in $\mathbb{R}^N$.

Since this holds for a.e.\ $t \in J$, we conclude
$u(x,t) = 0$ a.e.\ on $\mathbb{R}^N \times J$. 
\end{proof}

\begin{remark}\label{rem:gen}
    As is evident from the proof, one can add any number of local terms to the equation as long as the weak formulation remains well defined and the proof goes through verbatim after the terms so added disappear in the first step; one can similarly work with nonlinear time derivatives $\partial_t \phi(u)$ instead of $\partial_t u$ and it makes no difference after the first step. Finally, the technique leads to new unique continuation results for nonlocal kinetic equations which we do not pursue presently. 
\end{remark}

\begin{section}*{Acknowledgements}
   The author is grateful to Florian Grube for reviewing a preliminary version of the present manuscript. 
\end{section}

\bibliography{main}

\end{document}